\documentclass[reqno]{amsart}

\usepackage{enumitem,mathrsfs,mleftright}
\usepackage[colorlinks]{hyperref}

\newtheorem{Lem}{Lemma}
\newtheorem{Thm}{Theorem}

\theoremstyle{definition}
\newtheorem{Def}{Definition}
\newtheorem{Rmk}{Remark}

\setlist[enumerate]{leftmargin = *}
\setlist[itemize]{leftmargin = *}

\allowdisplaybreaks

\newcommand{\A}{\mathscr{A}}
\newcommand{\C}{\mathbb{C}}
\newcommand{\D}{\displaystyle}
\newcommand{\M}[1]{\func{M}{#1}}
\newcommand{\R}{\mathbb{R}}
\newcommand{\U}[1]{\func{\mathcal{U}}{#1}}
\newcommand{\Z}{\mathbb{Z}}
\newcommand{\Ad}[1]{\func{\operatorname{Ad}}{#1}}
\newcommand{\Br}[1]{\( #1 \)}
\newcommand{\Cc}[1]{\func{C_{c}}{#1}}
\newcommand{\df}{\stackrel{\operatorname{df}}{=}}
\newcommand{\Id}{\operatorname{Id}}
\newcommand{\om}[2]{\func{\omega}{#1,#2}}
\newcommand{\Aut}[1]{\func{\operatorname{Aut}}{#1}}
\newcommand{\Int}[4]{\int_{#1} #2 ~ \mathrm{d}{\func{#3}{#4}}}
\newcommand{\Map}[4]{\mleft\{ \begin{matrix} #1 & \to & #2 \\ #3 & \mapsto & \D #4 \end{matrix} \mright\}}
\newcommand{\Set}[2]{\mleft\{ #1 ~ \middle| ~ #2 \mright\}}
\newcommand{\func}[2]{#1 \Br{#2}}
\newcommand{\Func}[2]{\func{\Br{#1}}{#2}}
\newcommand{\FUNC}[2]{\func{\SqBr{#1}}{#2}}
\newcommand{\Norm}[1]{\mleft\| #1 \mright\|}
\newcommand{\Pair}[2]{\Br{#1,#2}}
\newcommand{\Quad}[4]{\Br{#1,#2,#3,#4}}
\newcommand{\SqBr}[1]{\[ #1 \]}
\newcommand{\SSet}[1]{\mleft\{ #1 \mright\}}
\newcommand{\Trip}[3]{\Br{#1,#2,#3}}
\newcommand{\Cstar}[1]{\func{C^{\ast}}{#1}}
\newcommand{\alphArg}[2]{\func{\alpha_{#1}}{#2}}
\newcommand{\alphExt}[1]{\overline{\alpha}_{#1}}
\newcommand{\alphArgExt}[2]{\func{\overline{\alpha}_{#1}}{#2}}

\renewcommand{\(}{\mleft(}
\renewcommand{\)}{\mright)}
\renewcommand{\[}{\mleft[}
\renewcommand{\]}{\mright]}
\renewcommand{\c}{\mathsf{c}}
\renewcommand{\u}{\operatorname{u}}

\begin{document}



\title[Automatic Continuity of $ \ast $-Representations]{Automatic Continuity of $ \ast $-Representations for Discrete Twisted $ C^{\ast} $-Dynamical Systems}
\author{Leonard T. Huang}
\email{Leonard.Huang@Colorado.EDU}
\address{Department of Mathematics \\ University of Colorado at Boulder \\ Campus Box 395 \\ Boulder \\ Colorado 80309 \\ United States of America}

\maketitle



\begin{abstract}
In this paper, we will establish the relatively unknown result that every $ \ast $-representation for a discrete twisted $ C^{\ast} $-dynamical system $ \Quad{G}{A}{\alpha}{\omega} $ is automatically contractive with respect to the $ L^{1} $-norm on $ \Cc{G,A} $.
\end{abstract}



\section{Introduction}


Given a $ C^{\ast} $-dynamical system $ \A = \Trip{G}{A}{\alpha} $ with a Haar measure $ \mu $ tacitly put on $ G $, we may equip the vector space $ \Cc{G,A} $ of compactly-supported continuous $ A $-valued functions on $ G $ with a $ \ast $-algebraic structure by defining a convolution $ \star_{\A} $ and an involution $ ^{\ast_{\A}} $ as follows:
\begin{align*}
\forall f,g \in \Cc{G,A}: \quad
f \star_{\A} g & \df \Map{G}{A}{x}{\Int{G}{\func{f}{y} \alphArg{y}{\func{g}{y^{- 1} x}}}{\mu}{y}}; \\
f^{\ast_{\A}}  & \df \Map{G}{A}{x}{\func{\Delta_{G}}{x^{- 1}} \cdot \alphArg{x}{\func{f}{x^{- 1}}^{\ast}}}.
\end{align*}
We may then turn $ \Trip{\Cc{G,A}}{\star_{\A}}{^{\ast_{\A}}} $ into a normed $ \ast $-algebra with a norm $ \Norm{\cdot}_{\A,1} $ on $ \Cc{G,A} $ defined by
$$
\forall f \in \Cc{G,A}: \qquad
\Norm{f}_{\A,1} \df \Int{G}{\Norm{\func{f}{x}}_{A}}{\mu}{x}.
$$
We call $ \Norm{\cdot}_{\A,1} $ the \emph{$ L^{1} $-norm} for $ \A $.

A $ \ast $-representation for $ \A $ is now a pair $ \Pair{C}{\pi} $ consisting of a $ C^{\ast} $-algebra $ C $ and an algebraic $ \ast $-homomorphism $ \pi $ from $ \Trip{\Cc{G,A}}{\star_{\A}}{^{\ast_{\A}}} $ to $ C $, and we may define the crossed-product $ C^{\ast} $-algebra $ \Cstar{\A} $ as the completion of $ \Trip{\Cc{G,A}}{\star_{\A}}{^{\ast_{\A}}} $ for a norm $ \Norm{\cdot}_{\A,\u} $ on $ \Cc{G,A} $ --- called the \emph{universal norm} for $ \A $ --- defined by
\begin{align*}
& \forall f \in \Cc{G,A}: \\
&     \Norm{f}_{\A,\u}
  \df \func{\sup}{
                 \Set{\Norm{\func{\pi}{f}}_{C}}
                     {
                     \begin{matrix}
                     \text{$ \Pair{C}{\pi} $ is a $ \ast $-representation for $ \A $ that is} \\
                     \text{contractive with respect to $ \Norm{\cdot}_{\A,1} $ and $ \Norm{\cdot}_{C} $}
                     \end{matrix}
                     }
                 }.
\end{align*}
As far as we know, all treatments of crossed-product $ C^{\ast} $-algebras (e.g. \cite{W}) assume the contractivity condition to \emph{enforce} that $ \Norm{\cdot}_{\A,\u} $ is actually well-defined.

We can now ask: Is a $ \ast $-representation for a $ C^{\ast} $-dynamical system automatically contractive with respect to $ \Norm{\cdot}_{\A,1} $ and $ \Norm{\cdot}_{C} $? We know of no counterexamples, and we have been unable to find anything relevant to this problem in the literature.

We hope to advertise the problem by showing that every $ \ast $-representation for a discrete $ C^{\ast} $-dynamical system is automatically contractive. Actually, we will prove a better result: Every $ \ast $-representation for a discrete \emph{twisted} $ C^{\ast} $-dynamical system is automatically contractive.

For every $ C^{\ast} $-algebra $ A $, we will adopt the following notation:
\begin{itemize}
\item
$ \Aut{A} $ denotes the group of $ \ast $-automorphisms on $ A $.

\item
$ \M{A} $ denotes the multiplier $ C^{\ast} $-algebra of $ A $.

\item
$ \U{A} $ denotes the group of unitary elements of $ A $.
\end{itemize}



\section{Twisted $ C^{\ast} $-Dynamical Systems}



\begin{Def}[\cite{BS}] \label{Twisted C*-Dynamical System}
A \emph{twisted $ C^{\ast} $-dynamical system} is a $ 4 $-tuple $ \Quad{G}{A}{\alpha}{\omega} $ with the following properties:
\begin{enumerate}
\item
$ G $ is a locally compact Hausdorff topological group, with a Haar measure $ \mu $ on $ G $ tacitly assumed.

\item
$ A $ is a $ C^{\ast} $-algebra.

\item
$ \alpha $ is a strongly continuous map from $ G $ to $ \Aut{A} $, i.e.,
$$
\Map{G}{A}{x}{\alphArg{x}{a}}
$$
is a continuous map for each $ a \in A $.

\item
$ \omega $ is a strictly continuous map from $ G \times G $ to $ \U{\M{A}} $, i.e.,
$$
\Map{G \times G}{A}{\Pair{x}{y}}{a \om{x}{y}}
\qquad \text{and} \qquad
\Map{G \times G}{A}{\Pair{x}{y}}{\om{x}{y} a}
$$
are continuous maps for each $ a \in A $.

\item
$ \alpha_{e} = \Id_{A} $, and $ \om{e}{r} = 1_{\M{A}} = \om{r}{e} $ for all $ r \in G $.

\item
$ \alphExt{r} \circ \alphExt{s} = \Ad{\om{r}{s}} \circ \alphExt{r s} $ for all $ r,s \in G $, i.e.,
$$
\forall m \in \M{A}: \quad
\alphArgExt{r}{\alphArgExt{s}{m}} = \om{r}{s} \alphArgExt{r s}{m} \om{r}{s}^{\ast}.
$$
Here, $ \overline{\alpha} $ denotes the map from $ G $ to $ \Aut{\M{A}} $ that assigns to each $ r \in G $ the unique $ \ast $-automorphism on $ \M{A} $ that continuously extends $ \alpha_{r} $.

\item
$ \alphArgExt{r}{\om{s}{t}} \om{r}{s t} = \om{r}{s} \om{r s}{t} $ for all $ r,s,t \in G $.
\end{enumerate}
If $ G $ is discrete, then we call $ \Quad{G}{A}{\alpha}{\omega} $ a \emph{discrete twisted $ C^{\ast} $-dynamical system}.
\end{Def}


For the rest of this paper, $ \A = \Quad{G}{A}{\alpha}{\omega} $ is a twisted $ C^{\ast} $-dynamical system.



\begin{Rmk}
Our definition of a twisted $ C^{\ast} $-dynamical system differs from that in \cite{BS}, which merely assumes that $ \alpha: G \to \Aut{A} $ is strongly Borel-measurable and $ \omega: G \times G \to \U{\M{A}} $ is strictly Borel-measurable. Such generality is not needed in our setting because we are only interested in continuous maps.
\end{Rmk}



\begin{Def}[\cite{BS}] \label{Convolution and Involution}
Define a convolution $ \star_{\A} $ and an involution $ ^{\ast_{\A}} $ on $ \Cc{G,A} $ by
\begin{align*}
\forall f,g \in \Cc{G,A}: \quad
f \star_{\A} g & \df \Map{G}{A}{x}{\Int{G}{\func{f}{y} \alphArg{y}{\func{g}{y^{- 1} x}} \om{y}{y^{- 1} x}}{\mu}{y}}; \\
f^{\ast_{\A}}  & \df \Map{G}{A}{x}{\func{\Delta_{G}}{x^{- 1}} \cdot \om{x}{x^{- 1}}^{\ast} \alphArg{x}{\func{f}{x^{- 1}}^{\ast}}}.
\end{align*}
\textbf{Note:} $ \Trip{\Cc{G,A}}{\star_{\A}}{^{\ast_{\A}}} $ is thus a $ \ast $-algebra.
\end{Def}



\begin{Def}
A \emph{$ \ast $-representation} for $ \A $ is a pair $ \Pair{C}{\pi} $, where $ C $ is a $ C^{\ast} $-algebra and $ \pi $ an algebraic $ \ast $-homomorphism from $ \Trip{\Cc{G,A}}{\star_{\A}}{^{\ast_{\A}}} $ to $ C $.
\end{Def}



\section{The Main Result}


In this section, $ \A $ is a discrete twisted $ C^{\ast} $-dynamical system.

By Haar's Theorem, the only Haar measures on $ G $ are positive scalar multiples of the counting measure $ \c $, so we fix $ k \in \R_{> 0} $ and tacitly assume the measure $ k \cdot \c $. By \autoref{Convolution and Involution}, we now have the following rules for convolution and involution:
\begin{align*}
\forall f,g \in \Cc{G,A}: \quad
f \star_{\A} g & \df \Map{G}{A}{x}{\sum_{y \in G} k \cdot \func{f}{y} \alphArg{y}{\func{g}{y^{- 1} x}} \om{y}{y^{- 1} x}}; \\
f^{\ast_{\A}}  & \df \Map{G}{A}{x}{\om{x}{x^{- 1}}^{\ast} \alphArg{x}{\func{f}{x^{- 1}}^{\ast}}}.
\end{align*}
Note that because $ G $ is discrete, it is unimodular, i.e., $ \Delta_{G} \equiv 1 $.



\begin{Def}
For each $ a \in A $ and $ r \in G $, define the function $ a \bullet \delta_{r} $ in $ \Cc{G,A} $ by
$$
\forall x \in G: \qquad
\Func{a \bullet \delta_{r}}{x} \df
\begin{cases}
a     & \text{if} ~ x = r; \\
0_{A} & \text{if} ~ x \in G \setminus \SSet{r}.
\end{cases}
$$
\end{Def}



\begin{Lem} \label{Important Convolutive and Involutive Identities}
The following identities hold:
\begin{enumerate}
\item
$ \Br{a \bullet \delta_{r}} \star_{\A} \Br{b \bullet \delta_{s}} = \Br{k \cdot a \alphArg{r}{b} \om{r}{s}} \bullet \delta_{r s} $ for all $ a,b \in A $ and $ r,s \in G $.

\item
$ \Br{a \bullet \delta_{r}}^{\ast_{\A}} = \om{r^{- 1}}{r}^{\ast} \alphArg{r^{- 1}}{a^{\ast}} \bullet \delta_{r^{- 1}} $ for all $ a \in A $ and $ r \in G $.

\item
$ \Br{a \bullet \delta_{e}} \star_{\A} \Br{b \bullet \delta_{e}} = \Br{k \cdot a b} \bullet \delta_{e} $ for all $ a,b \in A $.

\item
$ \Br{a \bullet \delta_{e}}^{\ast_{\A}} = a^{\ast} \bullet \delta_{e} $ for all $ a \in A $.
\end{enumerate}
\end{Lem}

\begin{proof}
It suffices to prove (1) and (2), as (3) and (4) are direct consequences.\footnote{To prove (3) and (4), we also require the normalizations in Property (5) of \autoref{Twisted C*-Dynamical System}.}

Let $ a,b \in A $ and $ r,s \in G $. Then
\begin{align*}
\forall x \in G: \quad
    \FUNC{\Br{a \bullet \delta_{r}} \star_{\A} \Br{b \bullet \delta_{s}}}{x}
& = \sum_{y \in G} k \cdot \Func{a \bullet \delta_{r}}{y} \alphArg{y}{\Func{b \bullet \delta_{s}}{y^{- 1} x}} \om{y}{y^{- 1} x} \\
& = k \cdot a \alphArg{r}{\Func{b \bullet \delta_{s}}{r^{- 1} x}} \om{r}{r^{- 1} x} \\
& = \begin{cases}
    k \cdot a \alphArg{r}{b} \om{r}{s} & \text{if} ~ x = r s; \\
    0_{A}                              & \text{if} ~ x \in G \setminus \SSet{r s}
    \end{cases} \\
& = \FUNC{\Br{k \cdot a \alphArg{r}{b} \om{r}{s}} \bullet \delta_{r s}}{x}; \\
    \func{\Br{a \bullet \delta_{r}}^{\ast_{\A}}}{x}
& = \om{x}{x^{- 1}}^{\ast} \alphArg{x}{\Func{a \bullet \delta_{r}}{x^{- 1}}^{\ast}} \\
& = \begin{cases}
    \om{r^{- 1}}{r}^{\ast} \alphArg{r^{- 1}}{a^{\ast}} & \text{if} ~ x = r^{- 1}; \\
    0_{A}                                              & \text{if} ~ x \in G \setminus \SSet{r^{- 1}}
    \end{cases} \\
& = \FUNC{\om{r^{- 1}}{r}^{\ast} \alphArg{r^{- 1}}{a^{\ast}} \bullet \delta_{r^{- 1}}}{x}.
\end{align*}
This completes the proof.
\end{proof}



\begin{Thm}
Let $ \Pair{C}{\pi} $ be a $ \ast $-representation for $ \A $. Then $ \Pair{C}{\pi} $ is automatically contractive with respect to $ \Norm{\cdot}_{\A,1} $ and $ \Norm{\cdot}_{C} $.
\end{Thm}

\begin{proof}
We will omit the subscript $ \A $ in $ \star_{\A} $, $ ^{\ast_{\A}} $, and $ \Norm{\cdot}_{\A,1} $ to ease notation.

Let $ a \in A $ and $ r \in G $. By \autoref{Important Convolutive and Involutive Identities}, we have
\begin{align*}
    \Br{a \bullet \delta_{r}}^{\ast} \star \Br{a \bullet \delta_{r}}
& = \Br{\om{r^{- 1}}{r}^{\ast} \alphArg{r^{- 1}}{a^{\ast}} \bullet \delta_{r^{- 1}}} \star \Br{a \bullet \delta_{r}} \\
& = \Br{k \cdot \om{r^{- 1}}{r}^{\ast} \alphArg{r^{- 1}}{a^{\ast}} \alphArg{r^{- 1}}{a} \om{r^{- 1}}{r}} \bullet \delta_{r^{- 1} r} \\
& = \Br{k \cdot \om{r^{- 1}}{r}^{\ast} \alphArg{r^{- 1}}{a^{\ast}} \alphArg{r^{- 1}}{a} \om{r^{- 1}}{r}} \bullet \delta_{e} \\
& = \Br{\om{r^{- 1}}{r}^{\ast} \alphArg{r^{- 1}}{a^{\ast}} \bullet \delta_{e}} \star \Br{\alphArg{r^{- 1}}{a} \om{r^{- 1}}{r} \bullet \delta_{e}} \\
& = \Br{\om{r^{- 1}}{r}^{\ast} \alphArg{r^{- 1}}{a}^{\ast} \bullet \delta_{e}} \star \Br{\alphArg{r^{- 1}}{a} \om{r^{- 1}}{r} \bullet \delta_{e}} \\
& = \Br{\SqBr{\alphArg{r^{- 1}}{a} \om{r^{- 1}}{r}}^{\ast} \bullet \delta_{e}} \star \Br{\alphArg{r^{- 1}}{a} \om{r^{- 1}}{r} \bullet \delta_{e}} \\
& = \Br{\alphArg{r^{- 1}}{a} \om{r^{- 1}}{r} \bullet \delta_{e}}^{\ast} \star \Br{\alphArg{r^{- 1}}{a} \om{r^{- 1}}{r} \bullet \delta_{e}}.
\end{align*}
Applying $ \pi $ to both ends and then using the $ C^{\ast} $-norm identity yields
\begin{equation} \label{First Norm Identity}
\Norm{\func{\pi}{a \bullet \delta_{r}}}_{C} = \Norm{\func{\pi}{\alphArg{r^{- 1}}{a} \om{r^{- 1}}{r} \bullet \delta_{e}}}_{C}.
\end{equation}
By \autoref{Important Convolutive and Involutive Identities} again, we have
\begin{align*}
  & ~ \Br{\alphArg{r^{- 1}}{a} \om{r^{- 1}}{r} \bullet \delta_{e}} \star \Br{\alphArg{r^{- 1}}{a} \om{r^{- 1}}{r} \bullet \delta_{e}}^{\ast} \\
= & ~ \Br{\alphArg{r^{- 1}}{a} \om{r^{- 1}}{r} \bullet \delta_{e}} \star \Br{\SqBr{\alphArg{r^{- 1}}{a} \om{r^{- 1}}{r}}^{\ast} \bullet \delta_{e}} \\
= & ~ \Br{\alphArg{r^{- 1}}{a} \om{r^{- 1}}{r} \bullet \delta_{e}} \star \Br{\om{r^{- 1}}{r}^{\ast} \alphArg{r^{- 1}}{a}^{\ast} \bullet \delta_{e}} \\
= & ~ \Br{k \cdot \alphArg{r^{- 1}}{a} \om{r^{- 1}}{r} \om{r^{- 1}}{r}^{\ast} \alphArg{r^{- 1}}{a}^{\ast}} \bullet \delta_{e} \\
= & ~ \Br{k \cdot \alphArg{r^{- 1}}{a} \alphArg{r^{- 1}}{a}^{\ast}} \bullet \delta_{e} \qquad
      \Br{\text{As $ \om{r^{- 1}}{r} $ is unitary.}} \\
= & ~ \Br{\alphArg{r^{- 1}}{a} \bullet \delta_{e}} \star \Br{\alphArg{r^{- 1}}{a}^{\ast} \bullet \delta_{e}} \\
= & ~ \Br{\alphArg{r^{- 1}}{a} \bullet \delta_{e}} \star \Br{\alphArg{r^{- 1}}{a} \bullet \delta_{e}}^{\ast}.
\end{align*}
Applying $ \pi $ to both ends and then using the $ C^{\ast} $-norm identity again yields
\begin{equation} \label{Second Norm Identity}
\Norm{\func{\pi}{\alphArg{r^{- 1}}{a} \om{r^{- 1}}{r} \bullet \delta_{e}}}_{C} = \Norm{\func{\pi}{\alphArg{r^{- 1}}{a} \bullet \delta_{e}}}_{C}.
\end{equation}
As $ a \in A $ and $ r \in G $ are arbitrary, we see from (\ref{First Norm Identity}) and (\ref{Second Norm Identity}) that
\begin{equation} \label{Main Norm Identity}
\forall a \in A, ~ \forall r \in G: \quad
\Norm{\func{\pi}{a \bullet \delta_{r}}}_{C} = \Norm{\func{\pi}{\alphArg{r^{- 1}}{a} \bullet \delta_{e}}}_{C}.
\end{equation}

Now, define for each $ r \in G $ a linear map $ \rho_{r}: A \to \Cc{G,A} $ by
$$
\forall a \in A: \quad
\func{\rho_{r}}{a} \df \Br{\frac{1}{k} \cdot \alphArg{r^{- 1}}{a}} \bullet \delta_{e}.
$$
Notice that \autoref{Important Convolutive and Involutive Identities} also gives us the following relations:
\begin{align*}
\forall a,b \in A, ~ \forall r \in G: \quad
    \func{\rho_{r}}{a b}
& = \Br{\frac{1}{k} \cdot \alphArg{r^{- 1}}{a b}} \bullet \delta_{e} \\
& = \Br{\frac{1}{k} \cdot \alphArg{r^{- 1}}{a} \alphArg{r^{- 1}}{b}} \bullet \delta_{e} \\
& = \frac{1}{k^{2}} \cdot \SqBr{\Br{k \cdot \alphArg{r^{- 1}}{a} \alphArg{r^{- 1}}{b}} \bullet \delta_{e}} \\
& = \frac{1}{k^{2}} \cdot \SqBr{\Br{\alphArg{r^{- 1}}{a} \bullet \delta_{e}} \star \Br{\alphArg{r^{- 1}}{b} \bullet \delta_{e}}} \\
& = \SqBr{\Br{\frac{1}{k} \cdot \alphArg{r^{- 1}}{a}} \bullet \delta_{e}} \star \SqBr{\Br{\frac{1}{k} \cdot \alphArg{r^{- 1}}{b}} \bullet \delta_{e}} \\
& = \func{\rho_{r}}{a} \star \func{\rho_{r}}{b}; \\
    \func{\rho_{r}}{a^{\ast}}
& = \Br{\frac{1}{k} \cdot \alphArg{r^{- 1}}{a^{\ast}}} \bullet \delta_{e} \\
& = \Br{\frac{1}{k} \cdot \alphArg{r^{- 1}}{a}^{\ast}} \bullet \delta_{e} \\
& = \Br{\SqBr{\frac{1}{k} \cdot \alphArg{r^{- 1}}{a}}^{\ast}} \bullet \delta_{e} \\
& = \SqBr{\Br{\frac{1}{k} \cdot \alphArg{r^{- 1}}{a}} \bullet \delta_{e}}^{\ast} \\
& = \func{\rho_{r}}{a}^{\ast}.
\end{align*}
Hence, $ \rho_{r} $ is an algebraic $ \ast $-homomorphism from $ A $ to $ \Trip{\Cc{G,A}}{\star}{^{\ast}} $ for all $ r \in G $, which makes $ \pi \circ \rho_{r} $ an algebraic $ \ast $-homomorphism from $ A $ to $ C $. As any algebraic $ \ast $-homomorphism from one $ C^{\ast} $-algebra to another is already contractive, we have
$$
\forall a \in A, ~ \forall r \in G: \quad
\Norm{\func{\pi}{\Br{\frac{1}{k} \cdot \alphArg{r^{- 1}}{a}} \bullet \delta_{e}}}_{C} \leq \Norm{a}_{A},
$$
or equivalently,
\begin{equation} \label{First Norm Inequality}
\forall a \in A, ~ \forall r \in G: \quad
\Norm{\func{\pi}{\alphArg{r^{- 1}}{a} \bullet \delta_{e}}}_{C} \leq k \Norm{a}_{A}.
\end{equation}
Combining (\ref{Main Norm Identity}) and (\ref{First Norm Inequality}) thus gives us
\begin{equation} \label{Main Norm Inequality}
\forall a \in A, ~ \forall r \in G: \quad
\Norm{\func{\pi}{a \bullet \delta_{r}}}_{C} \leq k \Norm{a}_{A}.
\end{equation}

Finally, let $ f \in \Cc{G,A} $. Then
\begin{align*}
       \Norm{\func{\pi}{f}}_{C}
& =    \Norm{\func{\pi}{\sum_{r \in G} \func{f}{r} \bullet \delta_{r}}}_{C} \\
& =    \Norm{\sum_{r \in G} \func{\pi}{\func{f}{r} \bullet \delta_{r}}}_{C} \\
& \leq \sum_{r \in G} \Norm{\func{\pi}{\func{f}{r} \bullet \delta_{r}}}_{C} \\
& \leq \sum_{r \in G} k \Norm{\func{f}{r}}_{A} \qquad
       \Br{\text{By (\ref{Main Norm Inequality}).}} \\
& =    \Norm{f}_{1}.
\end{align*}
Therefore, $ \pi $ is automatically contractive with respect to $ \Norm{\cdot}_{1} $ and $ \Norm{\cdot}_{C} $.
\end{proof}



\begin{Rmk}
If $ \pi $ is merely an algebraic homomorphism that does not respect the involution $ ^{\ast_{\A}} $, then continuity may fail spectacularly. Consider $ \A = \Quad{\Z}{\C}{\alpha^{0}}{\omega^{0}} $, where $ \alpha^{0} $ and $ \omega^{0} $ denote, respectively, the trivial action and the trivial multiplier. Suppose that $ \Z $ is equipped with the counting measure. Then the map
$$
\Map{\Cc{\Z}}{\C}{f}{\sum_{n \in \Z} \func{f}{n} e^{n}}
$$
is an unbounded algebraic homomorphism from $ \Pair{\Cc{\Z}}{\star_{\A}} $ to $ \C $, as $ \Norm{\delta_{n}}_{\A,1} = 1 $ for all $ n \in \Z $ but $ \D \lim_{n \to \infty} e^{n} = \infty $. It clearly \emph{does not} respect the involution $ ^{\ast_{\A}} $.
\end{Rmk}



\section{Conclusions}


The proof of the main result does not apply to other classes of locally compact Hausdorff groups, such as the abelian ones or the compact ones. One might work first on group $ C^{\ast} $-algebras instead of more general twisted $ C^{\ast} $-dynamical systems. Hopefully, the Peter-Weyl Theorem for compact groups and the Fourier transform for abelian groups could find a use, as they exploit the structure of these groups.



\section{Acknowledgments}


The author wishes to thank Dana Williams for helpful comments on this topic.




\end{document}